\documentclass[12pt]{article}

\usepackage{amssymb}
\usepackage{amsmath}
\usepackage{indentfirst,latexsym}
\usepackage{mathrsfs,hyperref}
\usepackage{graphicx}
\usepackage{cancel}
\usepackage{fancyhdr}
\usepackage{color}
\usepackage{amsthm,fullpage}
\usepackage{pgf,tikz}
\usepackage{mathrsfs}
\usetikzlibrary{arrows}

\tikzset{partition/.style={fill,circle,inner sep=1pt},
         part/.style={baseline=0,scale=0.5,bend left=45},
         partlabel/.style={below}}
\usetikzlibrary{shapes,arrows}
\tikzstyle{pnt}=[draw,ellipse,fill,inner sep=1pt]

\DeclareMathOperator\des{des}
\DeclareMathOperator\Des{Des}
\DeclareMathOperator\maj{maj}
\DeclareMathOperator\area{area}
\DeclareMathOperator\Exc{Exc}

\newtheorem{theorem}{Theorem}
\newtheorem{prop}[theorem]{Proposition}
\newtheorem{lemma}[theorem]{Lemma}
\newtheorem{corollary}[theorem]{Corollary}

\newcommand\beq{\begin{equation}}
\newcommand\eeq{\end{equation}}

\def\A{\mathcal A}
\def\bijsubs{f}
\def\bijhd{g}

\def\genbij{\theta}

\def\S{\mathcal S}
\def\C{\S^C}
\def\I{\mathcal I}
\def\CI{\I^C}

\def\Y{\mathcal Y}

\DeclareMathOperator\hd{HD}
 
\DeclareMathOperator\Peak{Peak}

\DeclareMathOperator\fp{fp}
\def\fn2{\lfloor \frac{n}{2} \rfloor}
\def\cn2{\lceil \frac{n}{2} \rceil}

\begin{document}

\author{Marilena Barnabei \\
Dipartimento di Matematica \\
Bologna, 40126, ITALY \\
\texttt{marilena.barnabei@unibo.it}\and
Flavio Bonetti \\
Dipartimento di Matematica \\
Bologna, 40126, ITALY \\
\texttt{flavio.bonetti@unibo.it}\and
Sergi Elizalde\thanks{Partially supported by grant \#280575 from the Simons Foundation and by grant H98230-14-1-0125 from the NSA.} \\
Department of Mathematics\\
Dartmouth College \\
Hanover, NH 03755, USA \\
\texttt{sergi.elizalde@dartmouth.edu}
\and
Matteo Silimbani \\
Dipartimento di Matematica \\
Bologna, 40126, ITALY \\
\texttt{matteo.silimbani4@unibo.it}}

\date{}

\title{Two descent statistics over $321$-avoiding centrosymmetric involutions}

\maketitle

\begin{abstract}
Centrosymmetric involutions in the symmetric group $\S_{2n}$ are permutations $\pi$ such that $\pi=\pi^{-1}$ and $\pi(i)+\pi(2n+1-i)=2n+1$ for all $i$, and they are in bijection with involutions of the hyperoctahedral group. We describe the distribution of some natural descent statistics on $321$-avoiding centrosymmetric involutions, including the number of descents in the first half of the involution, and the sum of the positions of these descents. Our results are based on two new bijections, one between
centrosymmetric involutions in $\S_{2n}$ and subsets of $\{1,\dots,n\}$, and another one showing that certain statistics on Young diagrams that fit inside a rectangle are equidistributed. We also use the latter bijection to refine a result from \cite{BBES} stating that the distribution of the major index on $321$-avoiding involutions is given by the $q$-analogue of the central binomial coefficients.
\end{abstract}



\section{Introduction}

The hyperoctahedral group $B_n$ is isomorphic to a subgroup of the symmetric group $\S_{2n}$, namely that of centrosymmetric permutations. We say that a permutation $\pi\in\S_{m}$ is {\em centrosymmetric} if $\pi(i)+\pi(m+1-i)=m+1$ for every $i=1,\ldots,m$. There are several ways to define a bijection between the set $B_n$ and the set of  centrosymmetric permutations in $\S_{2n}$, see e.g. \cite{Egge} for one such bijection.

Aside from their connection with $B_n$, centrosymmetric permutations are interesting in their own right. For instance, it is well known that a permutation is centrosymmetric if and only if the two standard Young tableaux corresponding to $\pi$ via the Robinson-Schensted algorithm are fixed under the Sch{\"u}tzenberger involution (see \cite{schutz} and \cite{KNU} for  details). In a different context, the permutation matrices corresponding to centrosymmetric permutations are the extreme points of a convex subset of $n^2$-dimensional Euclidean space, which is characterized in \cite{cro} by a simple set of linear inequalities.

In recent years, the study of forbidden patterns has been extended to the hyperoctahedral group $B_n$, the natural $B$-analogue of the symmetric group. Elements of $B_n$ are sometimes called signed permutations. In \cite{STE}, Stembridge 
used signed pattern avoidance to give a characterization of the so-called fully commutative top elements of the hyperoctahedral group, which
are elements having an interesting algebraic property. 
The sets of signed permutations avoiding signed patterns of length $2$ were completely characterized in \cite{SIM} and \cite{MW}, and the cardinalities of these sets were computed.

There is also some work in the literature on the distribution of permutation statistics both over centrosymmetric permutations and the hyperoctahedral group, as well as over subsets of these sets. For example, in \cite{bbs:old} the authors determine the descent distribution over the set of centrosymmetric involutions, while the same distribution over centrosymmetric permutations that avoid a pattern of length $3$ is given in \cite{BBS}. More recently, Biagioli et al. studied the distribution of the descent number and the major index both over the hyperoctahedral group \cite{biaze} and over the set of its fully commutative involutions \cite{BIJONA}. 

Involutions in $\S_m$ that avoid the pattern $321$ are particularly well behaved, and the distribution of descents sets over them has interesting connections with the theory of partitions, as shown in~\cite{BBES}.

In this paper we focus on centrosymmetric $321$-avoiding involutions, which we denote by $\CI_{m}(321)$, and we study the distribution of some descent statistics over them. Identifying $B_n$ with centrosymmetric permutations in $\S_{2n}$, the property of being an involution is preserved. The condition of avoiding $321$ can also be translated in terms of avoidance of some signed patterns in $B_n$.

On one hand, we determine the descent polynomial on $\CI_{2n}(321)$, showing that
$$\sum_{\pi\in \CI_{2n}(321)}q^{\des(\pi)}=(1+q)^n.$$
Recall that the major index is defined to be the sum of the descent positions of a permutation.  We observe that, if $\pi\in \CI_{m}(321)$ has a descent at position $i$, then it has also a descent at position $m-i$. Hence, the classical major index of $\pi$ is simply a multiple of its descent number.
For this reason, in addition to whole number of descents we consider two more statistics, which we denote $\des^+$ and $\maj^+$. They are defined as the number of descents in positions $1,\ldots,\lfloor m/2\rfloor$ and the sum of their positions, respectively. When $m$ is odd, these statistics reduce to known statistics on $321$-avoiding involutions studied in~\cite{BBES}, so we will focus on the case that $m=2n$.

An important ingredient in our study of these distributions is the surprising fact that every involution $\pi\in\CI_{2n}(321)$ is uniquely determined by its excedances in the first $n$ positions, and the descent set of $\pi$ can be easily read from these excedances. 
This will allow us to obtain the following generating polynomials:
$$\sum_{\pi\in \CI_{2n}(321)}q^{\des^+(\pi)}=\frac{(1+\sqrt{q})^{n+1}+(1-\sqrt{q})^{n+1}}{2}, \qquad
\sum_{\pi\in \CI_{2n}(321)}q^{\maj^+(\pi)}
=\sum_{h=0}^{n} q^{n-h}{n \choose h}_q,$$
where ${n \choose h}_q$ is the $q$-binomial coefficient. We give both recursive and bijective proofs of these results, exploiting the relationship between centrosymmetric involutions (more precisely, the sets of their excedances) and Young diagrams that fit inside a rectangular box. 

Our results translate easily to $123$-avoiding centrosymmetric involutions, since these are in bijection with $321$-avoiding ones via the complement operation.

One of the tools in our proofs is a bijection showing that certain statistics on Young diagrams that fit inside a rectangle are equidistributed.
In Section~\ref{sec:fp} we use this bijection to generalize the main result from~\cite{BBES},
which gives a bijection between $321$-avoiding involutions and partitions whose Young diagram fits into a square box, with the property that the descent set of the involution is mapped to the so-called hook decomposition of the partition. We show that by modifying the bijection and replacing square boxes by rectangles, one can additionally keep track of the number of fixed points. In particular, we refine a result from~\cite{BBES} stating that the distribution of the major index over the set of $321$-avoiding involutions is given by the $q$-analogue of the central binomial coefficients.

\section{Preliminaries}
\label{sectwo}

A permutation $\pi\in \S_m$ is called {\em centrosymmetric} if $$\pi(i)+\pi(m+1-i)=m+1$$
for every $1\le i\le m$. Equivalently, $\pi$ is centrosymmetric if $\pi^{r}=\pi^c$, where $r$ and $c$ are the usual reverse and complement operations, respectively.
We denote by $\C_m$ the set of centrosymmetric permutations in $\S_m$, and by $\CI_m$ the set of involutions in $\C_m$. 

We say that $\pi\in\S_m$ has a \emph{descent} at position $i$, where $1\le i<m$, if $\pi(i)>\pi(i+1)$. The set of descent positions of $\pi$ is denoted by $\Des(\pi)$.

\noindent Moreover, we denote by $\des(\pi)$ the cardinality of $\Des(\pi)$. The sum of the entries in $\Des(\pi)$ is called the \emph{major index} of $\pi$:
$$\maj(\pi)=\sum_{i\in \Des(\pi)}i.$$
For centrosymmetric permutations, the major index can be easily expressed in terms of the number of descents.
Observe that for $\pi\in\C_{m}$, we have that  $i\in\Des(\pi)$ if and only if $m-i\in\Des(\pi)$. This implies that $\maj(\pi)=m\des(\pi)/2$. Hence, studying the usual major index statistic on centrosymmetric permutations is equivalent to studying the number of descents.

Because of the above symmetry, it makes sense to restrict to the set of descents in positions $1,\ldots,\left\lfloor\frac{m}{2}\right\rfloor$ which allow us to recover the whole set $\Des(\pi)$. Setting $n=\left\lfloor\frac{m}{2}\right\rfloor$, we are interested in the statistics
$$\Des^+(\pi)=\Des(\pi)\cap [n], \qquad \des^+(\pi)=|\Des^+(\pi)|,$$ where we use the notation $[n]=\{1,2,\ldots,n\}$. We then have
\begin{equation}\label{eq:DesDes+} 
\Des(\pi)=\Des^+(\pi)\cup\{m-i:i\in \Des^+(\pi)\}.
\end{equation}
Note that this union is disjoint unless $m$ is even and $m/2\in \Des^+(\pi)$. 
It is now natural to define
$$\maj^+(\pi)=\sum_{i\in\Des^+(\pi)} i.$$

We say that a permutation $\pi\in \S_m$ {\em avoids} the pattern $\tau\in \S_k$ if $\pi$ does not contain a subsequence $\pi(i_1)\pi(i_2)\dots\pi(i_k)$ whose entries are in the same relative order as $\tau(1)\tau(2)\dots\tau(k)$. 
Denote by $\CI_{m}(321)$ the set of $321$-avoiding centrosymmetric involutions in $\S_{m}$. The cardinality of this set was found by Egge~\cite{Egge}, who showed that 
\begin{equation}\label{eq:odd}
|\CI_{2n+1}(321)|=\binom{n}{\fn2}
\end{equation}
and
\begin{equation}\label{eq:even}
|\CI_{2n}(321)|=2^n.
\end{equation}

The first main result of this paper gives the distribution of the statistics $\des^+$ and $\maj^+$ on $\CI_{m}(321)$.
When $m$ is odd, we will see that the distribution of these statistics on $321$-avoiding centrosymmetric involutions can be easily obtained from the results in \cite{BBES} about descents on $321$-avoiding involutions. For this reason, in this paper we will focus on the case when $m$ is even.
 
We point out that, in the even case, the proof given in~\cite{Egge} of Equation~\eqref{eq:even} is not bijective. In Theorem~\ref{cara} we will provide a bijective proof of this simple formula. This bijection will be a key ingredient in the proofs of the formulas giving the distribution of $\des^+$ and $\maj^+$ on $\CI_{2n}(321)$, which appear in Theorems~\ref{altromodo} and~\ref{initial}.

Another definition that will be useful is the notion of  \emph{excedance} of a permutation $\pi$, which is a position $i$ such that $\pi(i)>i$.
We denote by $\Exc(\pi)$ the set of excedances of $\pi$. We use the notation $\{a_1,a_2,\dots,a_r\}_<$ to indicate that $a_1<a_2<\dots<a_r$.

\section{The statistics $\des^+$ and $\maj^+$ on $\CI_{m}(321)$}
\label{main:results}

In this section we give formulas for the generating polynomials for the statistics $\des^+$, $\maj^+$ and $\des$ on $\CI_{m}(321)$. The next three theorems summarize the results in the case that $m$ is even.

\begin{theorem}\label{altromodo}
$$\sum_{\pi\in \CI_{2n}(321)}q^{\des^+(\pi)}=\sum_{k\geq 0} {n+1 \choose 2k}q^k=\frac{(1+\sqrt{q})^{n+1}+(1-\sqrt{q})^{n+1}}{2}.$$
\end{theorem}

\begin{theorem}\label{initial}
$$\sum_{\pi\in \CI_{2n}(321)}q^{\maj^+(\pi)}
=\sum_{h=0}^{n} {n \choose h}_q+(q^n-1)\sum_{h=0}^{n-1} {n-1 \choose h}_q=\sum_{h=0}^{n} q^{n-h}{n \choose h}_q,$$
where ${n \choose h}_q$ is the $q$-binomial coefficient. 
\end{theorem}

\begin{theorem}\label{desdistrib}
$$\sum_{\pi\in \CI_{2n}(321)}q^{\des(\pi)}=(1+q)^n.$$
\end{theorem}

We will give two different proofs of Theorems~\ref{altromodo} and~\ref{initial}: recursive ones in Section~\ref{sec:recursive} and bijective ones in Section~\ref{sec:bijective}. An important ingredient in all of them is the bijection that we describe in Section~\ref{sec:bij} between $\CI_{2n}(321)$ and the family of all subsets of $[n]$, which we denote by $2^{[n]}$. The same bijection will be used in Section~\ref{sec:proofdes} to prove Theorem~\ref{desdistrib}. 

When $m$ is odd, it is easy to check that every permutation $\pi\in \CI_{2n+1}(321)$ can be decomposed uniquely as $\pi=\alpha\,n{+}1\,\alpha'$, where $\alpha$ is an arbitrary element of $\I_n(321)$ and $\alpha'$ is the sequence $2n+2-\alpha(n),2n+2-\alpha(n-1),\dots,2n+2-\alpha(1)$. Thus, $\CI_{2n+1}(321)$ is in bijection with $\I_n(321)$, and Equation~\eqref{eq:odd} follows from~\cite{SS}.
Additionally, since
$$\des^+(\pi)=\des(\alpha),\quad\maj^+(\pi)=\maj(\alpha),\quad\mbox{and}\quad \des(\pi)=2\des(\alpha),$$
the formulas for the generating polynomials for $\des^+$, $\maj^+$ and $\des$ on $\CI_{2n+1}(321)$ follow from the results in \cite{BBES} about descents on $321$-avoiding involutions: 

\begin{prop}
\begin{align*}
\sum_{\pi\in \CI_{2n+1}(321)}q^{\des^+(\pi)}&=\sum_{k=0}^{\left\lfloor\frac{n}{2}\right\rfloor}{\left\lceil\frac{n}{2}\right\rceil\choose k}{\left\lfloor\frac{n}{2}\right\rfloor\choose k}q^k,\\
\sum_{\pi\in \CI_{2n+1}(321)}q^{\maj^+(\pi)}&={n \choose \left\lfloor\frac{n}{2}\right\rfloor}_q,\\
\sum_{\pi\in \CI_{2n+1}(321)}q^{\des(\pi)}&=\sum_{k=0}^{\left\lfloor\frac{n}{2}\right\rfloor}{\left\lceil\frac{n}{2}\right\rceil\choose k}{\left\lfloor\frac{n}{2}\right\rfloor\choose k}q^{2k}.
\end{align*}
\end{prop}

\subsection{A bijection $\CI_{2n}(321)\to 2^{[n]}$}\label{sec:bij}

Elements of $\CI_{2n}$ can be interpreted as symmetric matchings on $2n$ points. We draw matchings by placing $2n$ points on a horizontal line, labeled from $1$ to $2n$ from left to right, where some pairs of points are matched with an arc. Given $\pi\in\CI_{2n}$, the corresponding matching has an arc between $i$ and $j$ whenever $\pi(i)=j$ (equivalently, $\pi(j)=i$, since $\pi$ is an involution) and $i\neq j$. The centrosymmetric condition is equivalent to the matching being symmetric, that is, invariant under relabeling the points from right to left instead.

Under this interpretation, an involution avoids $321$ if and only if the corresponding matching is non-nesting, meaning that it does not contain two arcs $(i,l)$ and $(j,k)$ with $i<j<k<l$, or an arc $(i,k)$ and a singleton (i.e., a fixed point of the permutation) $j$ with $i<j<k$. Thus, elements of $\CI_{2n}(321)$ are in bijection with symmetric non-nesting matchings of $2n$ points.

For $\pi\in\CI_{2n}(321)$, define \begin{equation}\label{def:Epi} E_\pi=\Exc(\pi)\cap[n],\end{equation} that is, the set of excedances in the first $n$ positions.
If $M$ is the matching corresponding to $\pi$, let $E_M$ denote the set of values in $[n]$ that are matched with a larger value. Clearly, $E_\pi=E_M$. 

\begin{theorem}\label{cara} 
The map
$$\begin{array}{ccc}  \CI_{2n}(321) & \rightarrow & 2^{[n]} \\
\pi & \mapsto & E_\pi \end{array}$$
is a bijection. In particular, elements of $\CI_{2n}(321)$ are uniquely determined by their excedances in $[n]$.
\end{theorem}

\begin{proof}
By the above interpretation of $\CI_{2n}(321)$ as symmetric non-nesting matchings of $2n$ points, it will be enough to show that the map $M\mapsto E_M$ is a bijection between such matchings and $2^{[n]}$. We do so by describing its inverse. 

Given $E\subseteq[n]$, construct a matching $M_E$ as follows.
Read the elements of $E$ in increasing order, and for each $i\in E$ that has not been matched, draw an arc from $i$ to the smallest $j\notin E$ such that $j>i$ and $j$ has not been matched yet, and symmetrically draw an arc $(2n+1-j,2n+1-i)$, unless this is the same arc $(i,j)$ that was just added. 

For example for $n=11$ and $E=\{1,4,5,7,8,10\}\subseteq[11]$, we get the matching
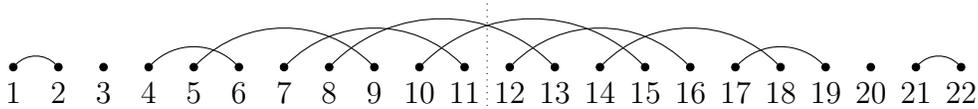
\begin{center}
\begin{tikzpicture}[scale=0.6]
   \foreach \i in {1,...,22}
        \node[pnt,label=below:$\i$] at (\i,0)(\i) {};
   \draw[dotted] (11.5,-1)--(11.5,1.5);
   \draw(1)  to [bend left=45] (2);
   \draw(4)  to [bend left=45] (6);
   \draw(5)  to [bend left=45] (9);
   \draw(7)  to [bend left=45] (11);
   \draw(8)  to [bend left=45] (13);
   \draw(10)  to [bend left=45] (15);
   \draw(12)  to [bend left=45] (16);
   \draw(14)  to [bend left=45] (18);
   \draw(17)  to [bend left=45] (19);
   \draw(21)  to [bend left=45] (22);
\end{tikzpicture}
\end{center}

Next we show that, for any $E\subseteq[n]$,  $M_E$ is a well-defined symmetric non-nesting matching.

The symmetry is clear because, at any stage of the process, when adding the arcs $(i,j)$ and $(2n+1-j,2n+1-i)$ (if different), the matching drawn so far is symmetric.

To see that it is well defined, we have to check that for every $i\in E$ that has not been matched, there is always an available vertex $j$ that it can be matched with. This is because since $i$ has not been matched, by symmetry neither has $2n+1-i$, and this value is not in $E$  because it is greater than $n$, so there is always at least one available vertex. This also proves that the arcs $(i,j)$ and $(2n+1-j,2n+1-i)$ added at each step satisfy $i+j\le 2n+1$, and thus $(2n+1-j)+(2n+1-i)\ge 2n+1$.

To show that $M_E$ is non-nesting, suppose first for contradiction that there is a triple $i<j<k$ where $(i,k)$ is an arc and $j$ is a singleton. By symmetry, we can assume without loss of generality that $i+k\le 2n+1$, so the arc $(i,k)$ was added as the first of a pair of symmetric arcs. But then, our construction would have matched $i$ with $j$, since $j\notin E$ (because it is a singleton) and $i<j<k$. 

Similarly, suppose for contradiction that there are two nested arcs $(i,l)$ and $(j,k)$ with $i<j<k<l$. Again, by symmetry, we can assume without loss of generality that $i+l\le 2n+1$, by flipping both arcs if necessary. In fact, we can also assume that  $j+k\le 2n+1$, because otherwise, flipping the arc $(j,k)$ we would still get two nested arcs.
But then, our construction would have matched $i$ with $k$, since $k\notin E$ (because it is matched with a smaller element) and the arc $(j,k)$ had not been added at that point of the process.

Finally, we show that the map $M\mapsto E_M$ is a bijection between the set of symmetric non-nesting matchings of $2n$ points and $2^{[n]}$, whose inverse is the map $E\mapsto M_E$.
To see that it is surjective, we show that for every $E\subseteq[n]$ we have $E_{M_E}=E$. In other words, we show that in $M_E$, the points in $[n]$ matched with larger points are precisely those in $E$. Clearly, the points in $E$ are always matched with larger points, since in our construction neither of the points $j$ and $2n+1-i$ in the pair of arcs $(i,j)$ and $(2n+1-j,2n+1-i)$  is in $E$. So it remains to show that if $2n+1-j\le n$, then $2n+1-j\in E$. This is because if $2n+1-j\notin E$, then, since $i<2n+1-j<j$, the point $i$ would have been matched with $2n+1-j$ rather than with $j$.

To see that the map $M\mapsto E_M$ is injective, we will show that for every $E\subseteq[n]$, the only symmetric non-nesting matching $M$ satisfying $E_{M}=E$ is the matching $M=M_E$. We will argue that at each step of the above construction that scans $E$ in increasing order, the only way to match $i\in E$ and preserve the non-nesting property is by matching it with the smallest $j\notin E$ such that $j>i$ and $j$ has not yet been matched. Suppose that this is not true, and consider the first time that we have some other choice for matching $i\in E$. It is clear that $i$ has to be matched with a vertex $j'\notin E$ such that $j'>i$ and $j'$ has not yet been matched, so suppose we match $i$ with a vertex $j'$ satisfying this property but not being the smallest. Then, we claim that the arc $(i,j')$ would be creating a nesting, which is a contradiction. Indeed, if $j$ does not become matched later on, then the singleton $j$ with the arc $(i,j')$ would form a nesting. If $j$ becomes matched with a point $r$, we know that $r>i$, since all the points in $E$ to the left of $i$ have been matched before $i$. If $i<r<j'$, then the arcs $(i,j')$ and $(r,j)$ (or $(j,r)$) would form a nesting. If $r>j'$, then the conditions $j\notin E$ and $E_{M}=E$ imply that $j>n$, and so the arcs $(i,j')$ and $(2n+1-r,2n+1-j)$ (which is forced by symmetry) would form an nesting, since $i<2n+1-r<2n+1-j<n<j'$, with the first inequality implied again by the fact that the points in $E$ to the left of $i$ have been matched before $i$. Clearly, once we add the arc $(i,j)$, the arc $(2n+1-j,2n+1-i)$ is forced by symmetry, so all the arcs in the construction of $M_E$ are forced.
\end{proof}

\subsection{Recursive proofs of Theorems~\ref{altromodo} and \ref{initial}}\label{sec:recursive}
The set $\Des^+(\pi)$ can be easily recovered from $E_{\pi}$ as follows.

\begin{lemma}\label{conno}
$\Des^+(\pi)=\{i\in [n]|i\in E_{\pi}\land i+1\notin E_{\pi}\}$. 
\end{lemma}

\begin{proof} Let $i\in [n-1]$, and consider four possible cases: 
\begin{itemize}
\item If $i\in E_\pi$ and $i+1\notin E_\pi$, then $\pi(i)>i$ and $\pi(i+1)\leq i+1$, so $i\in\Des^+(\pi)$. 
\item If $i\notin E_\pi$ and $i+1\in E_\pi$, then $\pi(i)\leq i$ and $\pi(i+1)> i+1$, so $i\notin\Des^+(\pi)$.
\item If $\{i,i+1\}\subseteq E_\pi$, then $\pi(i)>i$ and $\pi(i+1)>i+1$. If $\pi(i)>\pi(i+1)$, then $\pi$ would have the $321$ pattern 
$\pi(i)\;\pi(i{+}1)\; i{+}1$. Hence, $i\notin \Des^+(\pi)$.  
\item If neither $i$ nor $i+1$ are excedance positions, then $\pi(i)\leq i$ and $\pi(i+1)\leq i+1$. 
If $\pi(i)>\pi(i+1)$, then $\pi$ would have the $321$ pattern $i{+}1\; \pi(i)\;\pi(i{+}1)$. Hence, $i\notin \Des^+(\pi)$.  
\end{itemize}

Consider now the case $i=n$, and note that $n+1\notin E_\pi$. Since $\pi$ is centrosymmetric, $\pi(n)+\pi(n+1)=2n+1$. Thus, the condition for
$n\in\Des^+(\pi)$, that is, $\pi(n)>\pi(n+1)$, is equivalent to $\pi(n)>n$, which is the condition for $n\in E_\pi$.
\end{proof}

Lemma \ref{conno} suggests the following extension of the notion of descent set, descent number, and major index to subsets of $[n]$:
\begin{align*}\Des(E)&=\{i|i\in E\land i+1\notin E\},\\
\des(E)&=|\Des(E)|,\\
\maj(E)&=\sum_{i\in \Des(E)}i.
\end{align*}
With this notation, $\Des^+(\pi)=\Des(E_\pi)$, and so $\des^+(\pi)=\des(E_{\pi})$ and $\maj^+(\pi)=\maj(E_{\pi})$.

Define the the polynomials
$$d_n(q)=\sum_{\pi\in \CI_{2n}(321)}q^{\des^+(\pi)},\qquad p_n(q)=\sum_{\pi\in \CI_{2n}(321)}q^{\maj^+(\pi)}.$$
To prove Theorems~\ref{altromodo} and~\ref{initial}, we will show that $d_n(q)$ and $p_n(q)$ satisfy the same recurrences as the respective expressions given in these two theorems.

\begin{lemma}\label{boh}
For $n\geq 2$, the polynomials $d_n(q)$ and $p_n(q)$ satisfy the following recurrences:
$$d_n(q)=2d_{n-1}(q)+(q-1)d_{n-2}(q),$$
$$p_n(q)=(1+q)p_{n-1}(q)+(q^n-q)p_{n-2}(q).$$
\end{lemma}

\begin{proof} Theorem~\ref{cara}  and Lemma \ref{conno} allow us to rewrite the two polynomials $d_n(q)$ and $p_n(q)$ as follows:
$$d_n(q)=\sum_{E\subseteq [n]}q^{\des(E)},\qquad p_n(q)=\sum_{E\subseteq [n]}q^{\maj(E)}.$$
Consider now a subset $E\subseteq [n]$. There are three possible cases: 
\begin{itemize}
\item If $n\notin E$, then $E$ is a subset of $[n-1]$ with the same descent set and major index.
\item If $n\in E$ and $n-1\notin E$, then removing $n$ from $E$ yields a subset $E'$ of $[n-2]$ with $\Des(E')=\Des(E)\setminus \{n\}$ and $\maj(E')=\maj(E)-n$. 
\item If $n,n-1\in E$, then removing $n$ from $E$ yields a subset $E'\subseteq[n-1]$ with $\Des(E')=\Des(E)\setminus \{n\}\cup \{n-1\}$ and $\maj(E')=\maj(E)-1$. Since $E'$ contains $n-1$, the subsets $E'$ obtained in this way are those subsets of $[n-1]$ which are not a subsets of $[n-2]$.
\end{itemize}
These considerations give
$$d_n(q)=d_{n-1}(q)+qd_{n-2}(q)+(d_{n-1}(q)-d_{n-2}(q)),$$
$$p_n(q)=p_{n-1}(q)+q^np_{n-2}(q)+q(p_{n-1}(q)-p_{n-2}(q)),$$
which are equivalent to the above recurrences.
\end{proof}

 Theorem \ref{altromodo} follows now from Lemma \ref{boh}. Indeed, a routine computation shows that the polynomial 
$$\sum_{k\geq 0} {n+1 \choose 2k}q^k$$
satisfies the same recurrence as $d_n(q)$ with the same initial conditions.

 In order to prove Theorem \ref{initial}, we need a few more definitions. Let $R_{a,b}$ denote an $a\times b$ rectangle, and let $\Y_{a,b}$ denote the set of Young diagrams that fit inside of $R_{a,b}$ (where the convention is that diagrams are placed touching the upper and left sides of the rectangle). 
By considering the southeast boundary of the Young diagram, we identify $\Y_{a,b}$ with the set of paths with steps $N$ and $E$ from $(0,0)$ to $(b,a)$. It is well known that
$|\Y_{a,b}|=\binom{a+b}{a}$. 

 For every positive integer $n$, we can represent the subsets of $[n]$ as lattice paths with unit steps $N$ and $E$ starting at the origin, where the $i$th step is $N$ if and only if $i$ belongs to the subset. Let $\A_n$ denote the set of all lattice paths with $n$ steps $N$ and $E$ starting at the origin. Note that $|\A_n|=2^n$ and
$$\A_n=\bigcup_{a+b=n} \Y_{a,b}.$$

 Consider now the set $\Y^E_{a,b}$ of paths with steps $N$ and $E$ from $(0,0)$ to $(b+1,a)$ starting with an $E$ step. The set $\Y^E_{a,b}$  is obviously in bijection with $\Y_{a,b}$.
Define
$$\A^E_n=\bigcup_{a+b=n} \Y
^E_{a,b}$$
and
\begin{equation}\label{def:r} r_n(q)=\sum_{P\in \A^E_n}q^{\area(P)},\end{equation}
where $\area(P)$ is the area of the Young diagram corresponding to $P$. Then,
$$r_n(q)=\sum_{a=0}^n\sum_{P\in \Y^E_{a,n-a}}q^{\area(P)}=\sum_{a=0}^nq^{n-a}\sum_{\bar{P}\in \Y_{a,n-a}}q^{\area(\bar{P})},$$
where $\bar{P}$ is the path obtained from $P$ by deleting the first step.
This last espression equals $$\sum_{a=0}^{n} q^{n-a}{n \choose a}_q,$$
using the well-known fact~\cite{Sta} that the coefficient of $q^r$ in the q-binomial coefficient ${n\choose a}_q$ equals the number of partitions of $r$ whose Young diagram is in $\Y_{a,n-a}$.

To prove Theorem \ref{initial}, it is now enough to show is that the polynomials $r_n(q)$ satisfy the same recurrence as the one given in Lemma~\ref{boh} for $p_n(q)$. It is clear that their initial values coincide.

\begin{prop}\label{prop:recr}
For $n\geq 2$, the polynomials $r_n(q)$ defined in Equation~\eqref{def:r} satisfy the recurrence
$$r_n(q)=(1+q)r_{n-1}(q)+(q^n-q)r_{n-2}(q).$$
\end{prop}

\begin{proof} Consider a path $P\in \A^E_n$. Once again, there are three possible cases, as illustrated in Figure~\ref{fig:r}:
\begin{itemize}
\item If $P$ ends with an $E$ step, remove it, getting an element of $\A^E_{n-1}$ with the same area.
\item If $P$ ends with an $N$ step and begins with $EE$, remove the first and last step from $P$, getting an element of $\A^E_{n-2}$ whose area has decreased by $n$. 
\item If $P$ ends with an $N$ step and begins with $EN$, remove the second step in $P$, getting an element of $\A^E_{n-1}$ whose area has decreased by $1$. The resulting path is not an arbitrary element of $\A^E_{n-1}$, but one ending with an $N$ step. Such paths contribute $r_{n-1}(q)-r_{n-2}(q)$ to the generating polynomial, since paths in $\A^E_{n-1}$ ending with an $E$ step are equivalent to paths in $\A^E_{n-2}$.
\end{itemize}
These considerations give
$$r_n(q)=r_{n-1}(q)+q^nr_{n-2}(q)+q(r_{n-1}(q)-r_{n-2}(q)),$$
which is equivalent to the stated recurrence.
\end{proof}

\begin{figure}[htb]
\centering
\definecolor{ffqqqq}{rgb}{1.,0.,0.}
\definecolor{qqqqff}{rgb}{0.,0.,1.}
\definecolor{ffffqq}{rgb}{1.,1.,0.}
\begin{tikzpicture}[line cap=round,line join=round,>=triangle 45,x=4mm,y=4mm]
\fill[color=ffffqq,fill=ffffqq,fill opacity=0.12] (0.,1) -- (5.,1) -- (5.,5.976002603987413) -- (0.,5.976002603987413) -- cycle;
\fill[color=ffffqq,fill=ffffqq,fill opacity=0.12] (9.,1.) -- (13.,1.) -- (13.,6.) -- (9.,6.) -- cycle;
\fill[color=ffqqqq,fill=ffqqqq,fill opacity=0.1] (0.,0.9760026039874132) -- (1.,0.9760026039874132) -- (1.,4.) -- (3.,4.) -- (3.,5.976002603987413) -- (0.,5.976002603987413) -- cycle;
\fill[color=ffqqqq,fill=ffqqqq,fill opacity=0.1] (9.,1.) -- (10.,1.) -- (10.,4.) -- (12.,4.) -- (12.,6.) -- (9.,6.) -- cycle;
\fill[color=ffffqq,fill=ffffqq,fill opacity=0.12] (0.,-8.) -- (5.,-8.) -- (5.,-2.) -- (0.,-2.) -- cycle;
\fill[color=ffqqqq,fill=ffqqqq,fill opacity=0.1] (0.,-8.) -- (2.,-8.) -- (2.,-4.) -- (5.,-4.) -- (5.,-2.) -- (0.,-2.) -- cycle;
\fill[color=ffffqq,fill=ffffqq,fill opacity=0.12] (10.,-8.) -- (14.,-8.) -- (14.,-3.) -- (10.,-3.) -- cycle;
\fill[color=ffqqqq,fill=ffqqqq,fill opacity=0.1] (10.,-8.) -- (11.,-8.) -- (11.,-4.) -- (14.,-4.) -- (14.,-3.) -- (10.,-3.) -- cycle;
\fill[color=ffffqq,fill=ffffqq,fill opacity=0.12] (0.,-16.) -- (5.,-16.) -- (5.,-10.) -- (0.,-10.) -- cycle;
\fill[color=ffqqqq,fill=ffqqqq,fill opacity=0.1] (0.,-16.) -- (1.,-16.) -- (1.,-13.) -- (3.,-13.) -- (3.,-12.) -- (5.,-12.) -- (5.,-10.) -- (0.,-10.) -- cycle;
\fill[color=ffffqq,fill=ffffqq,fill opacity=0.12] (8.990734685700113,-15.008306596296148) -- (13.990734685700113,-15.008306596296148) -- (14.,-10.) -- (9.,-10.) -- cycle;
\fill[color=ffqqqq,fill=ffqqqq,fill opacity=0.1] (8.990761988121974,-14.993548446158929) -- (9.990734685700113,-15.008306596296148) -- (10.,-13.) -- (12.,-13.) -- (12.,-12.) -- (14.,-12.) -- (14.,-10.) -- (9.,-10.) -- cycle;
\draw [line width=2.8pt,color=ffffqq] (0.,1)-- (5.,1);
\draw [line width=2.8pt,color=ffffqq] (5.,1)-- (5.,5.976002603987413);
\draw [line width=2.8pt,color=ffffqq] (5.,5.976002603987413)-- (0.,5.976002603987413);
\draw [line width=2.8pt,color=ffffqq] (0.,5.976002603987413)-- (0.,1);
\draw [line width=2.8pt,color=qqqqff] (0.,0.9760026039874132)-- (1.,0.9760026039874132);
\draw [line width=2.8pt,color=qqqqff] (1.,0.9760026039874132)-- (1.,2.);
\draw [line width=2.8pt,color=qqqqff] (1.,2.)-- (1.,3.);
\draw [line width=2.8pt,color=qqqqff] (1.,3.)-- (1.,4.);
\draw [line width=2.8pt,color=qqqqff] (1.,4.)-- (2.,4.);
\draw [line width=2.8pt,color=qqqqff] (2.,4.)-- (3.,4.);
\draw [line width=2.8pt,color=qqqqff] (3.,4.)-- (3.,5.);
\draw [line width=2.8pt,color=qqqqff] (3.,5.)-- (3.,5.976002603987413);
\draw [line width=2.8pt,color=qqqqff] (3.,5.976002603987413)-- (4.,5.976002603987413);
\draw [line width=2.8pt,color=qqqqff] (4.,5.976002603987413)-- (5.,5.976002603987413);
\draw [line width=2.8pt,color=ffffqq] (9.,1.)-- (13.,1.);
\draw [line width=2.8pt,color=ffffqq] (13.,1.)-- (13.,6.);
\draw [line width=2.8pt,color=ffffqq] (13.,6.)-- (9.,6.);
\draw [line width=2.8pt,color=ffffqq] (9.,6.)-- (9.,1.);
\draw [line width=2.8pt,color=qqqqff] (9.,1.)-- (10.,1.);
\draw [line width=2.8pt,color=qqqqff] (10.,1.)-- (10.,2.);
\draw [line width=2.8pt,color=qqqqff] (10.,2.)-- (10.,3.);
\draw [line width=2.8pt,color=qqqqff] (10.,3.)-- (10.,4.);
\draw [line width=2.8pt,color=qqqqff] (10.,4.)-- (11.,4.);
\draw [line width=2.8pt,color=qqqqff] (11.,4.)-- (12.,4.);
\draw [line width=2.8pt,color=qqqqff] (12.,4.)-- (12.,5.);
\draw [line width=2.8pt,color=qqqqff] (12.,5.)-- (12.,6.);
\draw [line width=2.8pt,color=qqqqff] (12.,6.)-- (13.,6.);
\draw [color=ffqqqq] (0.,0.9760026039874132)-- (1.,0.9760026039874132);
\draw [color=ffqqqq] (1.,0.9760026039874132)-- (1.,4.);
\draw [color=ffqqqq] (1.,4.)-- (3.,4.);
\draw [color=ffqqqq] (3.,4.)-- (3.,5.976002603987413);
\draw [color=ffqqqq] (3.,5.976002603987413)-- (0.,5.976002603987413);
\draw [color=ffqqqq] (0.,5.976002603987413)-- (0.,0.9760026039874132);
\draw [color=ffqqqq] (9.,1.)-- (10.,1.);
\draw [color=ffqqqq] (10.,1.)-- (10.,4.);
\draw [color=ffqqqq] (10.,4.)-- (12.,4.);
\draw [color=ffqqqq] (12.,4.)-- (12.,6.);
\draw [color=ffqqqq] (12.,6.)-- (9.,6.);
\draw [color=ffqqqq] (9.,6.)-- (9.,1.);
\draw [->,line width=0.4pt] (6.203831553619772,2.9891957286874513) -- (7.703831553619772,2.9891957286874513);
\draw [line width=2.8pt,color=ffffqq] (0.,-8.)-- (5.,-8.);
\draw [line width=2.8pt,color=ffffqq] (5.,-8.)-- (5.,-2.);
\draw [line width=2.8pt,color=ffffqq] (5.,-2.)-- (0.,-2.);
\draw [line width=2.8pt,color=ffffqq] (0.,-2.)-- (0.,-8.);
\draw [line width=2.8pt,color=qqqqff] (0.,-8.)-- (0.,-8.);
\draw [line width=2.8pt,color=qqqqff] (0.,-8.)-- (2.,-8.);
\draw [line width=2.8pt,color=qqqqff] (2.,-8.)-- (2.,-6.);
\draw [line width=2.8pt,color=qqqqff] (2.,-6.)-- (2.,-5.);
\draw [line width=2.8pt,color=qqqqff] (2.,-5.)-- (2.,-4.);
\draw [line width=2.8pt,color=qqqqff] (2.,-4.)-- (2.,-4.);
\draw [line width=2.8pt,color=qqqqff] (2.,-4.)-- (5.,-4.);
\draw [color=ffqqqq] (0.,-8.)-- (2.,-8.);
\draw [color=ffqqqq] (2.,-8.)-- (2.,-4.);
\draw [color=ffqqqq] (2.,-4.)-- (5.,-4.);
\draw [color=ffqqqq] (5.,-4.)-- (5.,-2.);
\draw [color=ffqqqq] (5.,-2.)-- (0.,-2.);
\draw [color=ffqqqq] (0.,-2.)-- (0.,-8.);
\draw [->,line width=0.4pt] (6.203831553619769,-4.986806875299962) -- (7.703831553619761,-4.986806875299962);
\draw [->,line width=0.4pt] (6.203831553619768,-12.986806875299962) -- (7.70383155361977,-12.986806875299962);
\draw [line width=2.8pt,color=qqqqff] (5.,-4.)-- (4.,-4.);
\draw [line width=2.8pt,color=qqqqff] (4.,-4.)-- (5.,-4.);
\draw [line width=2.8pt,color=qqqqff] (5.,-4.)-- (5.,-3.);
\draw [line width=2.8pt,color=qqqqff] (5.,-3.)-- (5.,-2.);
\draw [line width=2.8pt,color=qqqqff] (2.,-4.)-- (5.,-4.);
\draw [line width=2.8pt,color=ffffqq] (10.,-8.)-- (14.,-8.);
\draw [line width=2.8pt,color=ffffqq] (14.,-8.)-- (14.,-3.);
\draw [line width=2.8pt,color=ffffqq] (14.,-3.)-- (10.,-3.);
\draw [line width=2.8pt,color=ffffqq] (10.,-3.)-- (10.,-8.);
\draw [line width=2.8pt,color=qqqqff] (10.,-8.)-- (10.,-8.);
\draw [line width=2.8pt,color=qqqqff] (11.,-8.)-- (11.,-6.);
\draw [line width=2.8pt,color=qqqqff] (11.,-6.)-- (11.,-5.);
\draw [line width=2.8pt,color=qqqqff] (11.,-5.)-- (11.,-4.);
\draw [line width=2.8pt,color=qqqqff] (11.,-4.)-- (11.,-4.);
\draw [line width=2.8pt,color=qqqqff] (11.,-4.)-- (14.,-4.);
\draw [color=ffqqqq] (10.,-8.)-- (11.,-8.);
\draw [color=ffqqqq] (11.,-8.)-- (11.,-4.);
\draw [color=ffqqqq] (11.,-4.)-- (14.,-4.);
\draw [color=ffqqqq] (14.,-4.)-- (14.,-3.);
\draw [color=ffqqqq] (14.,-3.)-- (10.,-3.);
\draw [color=ffqqqq] (10.,-3.)-- (10.,-8.);
\draw [line width=2.8pt,color=qqqqff] (14.,-4.)-- (13.,-4.);
\draw [line width=2.8pt,color=qqqqff] (13.,-4.)-- (14.,-4.);
\draw [line width=2.8pt,color=qqqqff] (14.,-4.)-- (14.,-3.);
\draw [line width=2.8pt,color=qqqqff] (11.,-4.)-- (14.,-4.);
\draw [line width=2.8pt,color=ffffqq] (0.,-16.)-- (5.,-16.);
\draw [line width=2.8pt,color=ffffqq] (5.,-16.)-- (5.,-10.);
\draw [line width=2.8pt,color=ffffqq] (5.,-10.)-- (0.,-10.);
\draw [line width=2.8pt,color=ffffqq] (0.,-10.)-- (0.,-16.);
\draw [line width=2.8pt,color=qqqqff] (0.,-16.)-- (0.,-16.);
\draw [line width=2.8pt,color=qqqqff] (0.,-16.)-- (1.,-16.);
\draw [line width=2.8pt,color=qqqqff] (1.,-16.)-- (1.,-14.);
\draw [line width=2.8pt,color=qqqqff] (1.,-14.)-- (1.,-13.);
\draw [line width=2.8pt,color=qqqqff] (1.,-13.)-- (1.,-13.);
\draw [line width=2.8pt,color=qqqqff] (1.,-13.)-- (3.,-13.);
\draw [line width=2.8pt,color=qqqqff] (5.,-12.)-- (5.,-11.);
\draw [line width=2.8pt,color=qqqqff] (5.,-11.)-- (5.,-10.);
\draw [line width=2.8pt,color=qqqqff] (3.,-13.)-- (3.,-12.);
\draw [line width=2.8pt,color=qqqqff] (3.,-12.)-- (5.,-12.);
\draw [color=ffqqqq] (0.,-16.)-- (1.,-16.);
\draw [color=ffqqqq] (1.,-16.)-- (1.,-13.);
\draw [color=ffqqqq] (1.,-13.)-- (3.,-13.);
\draw [color=ffqqqq] (3.,-13.)-- (3.,-12.);
\draw [color=ffqqqq] (3.,-12.)-- (5.,-12.);
\draw [color=ffqqqq] (5.,-12.)-- (5.,-10.);
\draw [color=ffqqqq] (5.,-10.)-- (0.,-10.);
\draw [color=ffqqqq] (0.,-10.)-- (0.,-16.);
\draw [line width=2.8pt,color=qqqqff] (10.,-14.)-- (10.,-13.);
\draw [line width=2.8pt,color=qqqqff] (10.,-13.)-- (10.,-13.);
\draw [line width=2.8pt,color=qqqqff] (10.,-13.)-- (12.,-13.);
\draw [line width=2.8pt,color=qqqqff] (14.,-12.)-- (14.,-11.);
\draw [line width=2.8pt,color=qqqqff] (14.,-11.)-- (14.,-10.);
\draw [line width=2.8pt,color=qqqqff] (12.,-13.)-- (12.,-12.);
\draw [line width=2.8pt,color=qqqqff] (12.,-12.)-- (14.,-12.);
\draw [line width=2.8pt,color=qqqqff] (10.,-15.)-- (10.,-14.);
\draw [line width=2.8pt,color=ffffqq] (8.990734685700113,-15.008306596296148)-- (13.990734685700113,-15.008306596296148);
\draw [line width=2.8pt,color=ffffqq] (13.990734685700113,-15.008306596296148)-- (14.,-10.);
\draw [line width=2.8pt,color=ffffqq] (14.,-10.)-- (9.,-10.);
\draw [line width=2.8pt,color=ffffqq] (9.,-10.)-- (8.990734685700113,-15.008306596296148);
\draw [line width=2.8pt,color=qqqqff] (8.990761988121974,-14.993548446158929)-- (9.990734685700113,-15.008306596296148);
\draw [line width=2.8pt,color=qqqqff] (14.,-12.)-- (14.,-11.);
\draw [line width=2.8pt,color=qqqqff] (14.,-11.)-- (14.,-10.);
\draw [color=ffqqqq] (8.990761988121974,-14.993548446158929)-- (9.990734685700113,-15.008306596296148);
\draw [color=ffqqqq] (9.990734685700113,-15.008306596296148)-- (10.,-13.);
\draw [color=ffqqqq] (10.,-13.)-- (12.,-13.);
\draw [color=ffqqqq] (12.,-13.)-- (12.,-12.);
\draw [color=ffqqqq] (12.,-12.)-- (14.,-12.);
\draw [color=ffqqqq] (14.,-12.)-- (14.,-10.);
\draw [color=ffqqqq] (14.,-10.)-- (9.,-10.);
\draw [color=ffqqqq] (9.,-10.)-- (8.990761988121974,-14.993548446158929);
\draw [line width=2.8pt,color=qqqqff] (10.000027302421861,-7.985241849862781)-- (11.,-8.);
\draw (1.,4.)-- (1.,5.976002603987413);
\draw (2.,4.)-- (2.,6.);
\draw (3.,5.)-- (0.,5.);
\draw (1.,4.)-- (0.,4.);
\draw (1.,3.)-- (0.,3.);
\draw (1.,2.)-- (0.,2.);
\draw (12.,5.)-- (9.,5.);
\draw (10.,4.)-- (9.,4.);
\draw (10.,3.)-- (9.,3.);
\draw (10.,2.)-- (9.,2.);
\draw (11.,-4.)-- (10.,-4.);
\draw (11.,-5.)-- (10.,-5.);
\draw (11.,-6.)-- (10.,-6.);
\draw (11.,-7.)-- (10.,-7.);
\draw (1.,-14.)-- (0.,-14.);
\draw (1.,-15.)-- (0.,-15.);
\draw (10.,-14.)-- (9.,-14.);
\draw (1.,-13.)-- (0.,-13.);
\draw (10.,-13.)-- (9.,-13.);
\draw (10.,4.)-- (10.,5.976002603987414);
\draw (11.,4.)-- (11.,5.976002603987414);
\draw (4.,-4.)-- (4.,-2.0239973960125868);
\draw (3.,-4.)-- (3.,-2.0239973960125868);
\draw (2.,-4.)-- (2.,-2.0239973960125868);
\draw (4.,-12.)-- (4.,-10.023997396012586);
\draw (3.,-12.)-- (3.,-10.023997396012586);
\draw (12.,-12.)-- (12.,-10.023997396012586);
\draw (13.,-12.)-- (13.,-10.023997396012586);
\draw (1.,-8.)-- (1.,-2.);
\draw (5.,-3.)-- (0.,-3.);
\draw (2.,-4.)-- (0.,-4.);
\draw (2.,-5.)-- (0.,-5.);
\draw (2.,-6.)-- (0.,-6.);
\draw (2.,-7.)-- (0.,-7.);
\draw (11.,-4.)-- (11.,-3.);
\draw (12.,-4.)-- (12.,-3.);
\draw (13.,-4.)-- (13.,-3.);
\draw (1.,-13.)-- (1.,-10.);
\draw (2.,-13.)-- (2.,-10.);
\draw (10.,-13.)-- (10.,-10.);
\draw (11.,-13.)-- (11.,-10.);
\draw (5.,-11.)-- (0.,-11.);
\draw (14.,-11.)-- (9.,-11.);
\draw (3.,-12.)-- (0.,-12.);
\draw (12.,-12.)-- (9.,-12.);
\end{tikzpicture}
\caption{The three cases in the proof of Proposition~\ref{prop:recr}.}
\label{fig:r}
\end{figure}
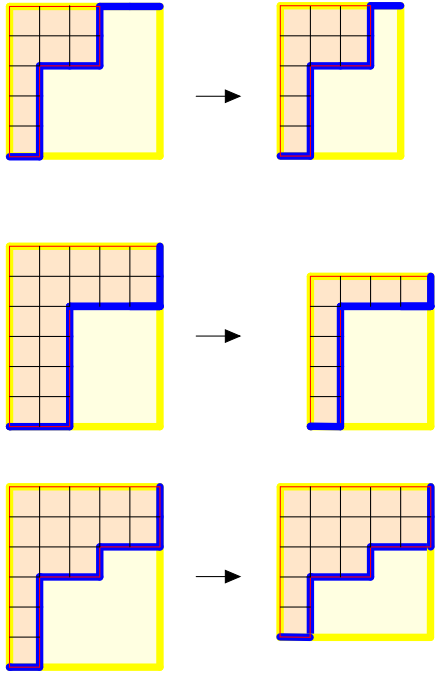

\subsection{Proof of Theorem~\ref{desdistrib}}\label{sec:proofdes}

To obtain the generating polynomial for the statistic $\des$ on $\CI_{2n}(321)$, we will combine Lemma \ref{conno} with the following simple fact. 
\begin{lemma}\label{des} Let $\pi \in \CI_{2n}(321)$ and recall the definition of $E_{\pi}$ from Equation~\eqref{def:Epi}. Then,
$$\des(\pi)=\begin{cases} 2\des^+(\pi)-1=2\des(E_{\pi})-1  & \mbox{if $n\in E_{\pi}$,} \\  2\des^+(\pi)=2\des(E_{\pi})   & \mbox{otherwise.} \end{cases}$$. 

\end{lemma}   
\begin{proof} By Equation~\eqref{eq:DesDes+},
$\Des(\pi)=\Des^+(\pi)\cup\{2n-i:i\in \Des^+(\pi)\}$
and this union is disjoint unless  $n\in \Des^+(\pi)$.
\end{proof}

\begin{proof}[Proof of Theorem~\ref{desdistrib}]
By Lemma \ref{conno}, involutions in $\CI_{2n}(321)$ are in bijection with subsets of $[n]$.
Let $\pi\in\CI_{2n-2}(321)$ and let $E=E_\pi\subseteq[n-1]$ be the corresponding subset. We build two involutions
in $\CI_{2n}(321)$ as follows: let $\tilde{\pi}$ be the involution corresponding to $E$, seen as a subset of $[n]$; and let $\hat{\pi}$ be the involution corresponding to the set $E\cup \{n\}$. Setting $t=\des(E)$ and applying Lemma \ref{des}, we have
\begin{itemize}
\item If $n-1 \in E$, then 
$$ \des(\pi)=2t-1, \quad \quad \des(\tilde{\pi})=2t, \quad \quad \des(\hat{\pi})=2t-1.$$
\item If $n-1 \notin E$, then 
$$ \des(\pi)=2t, \quad \quad \des(\tilde{\pi})=2t, \quad \quad \des(\hat{\pi})=2t+1.$$
\end{itemize}
This implies that
$$\sum_{\pi\in \CI_{2n}(321)}q^{\des(\pi)}=(1+q)\sum_{\pi\in \CI_{2n-2}(321)}q^{\des(\pi)}.$$
Since $\sum_{\pi\in \CI_{2}(321)}q^{\des(\pi)}=1+q$, we get formula by induction.
\end{proof}

\section{Bijective proofs of Theorems~\ref{altromodo} and \ref{initial}}\label{sec:bijective}

 In this section we show that Theorems~\ref{altromodo} and \ref{initial} can be also proved bijectively. These bijective proofs are based upon a more careful analysis of the connections between lattice paths and integer partitions.

 Define a {\em peak} of a path in $\A_n$ to be an occurrence of $NE$, or equivalently the vertex in the middle of such an occurrence.
If we label the vertices of a path $P\in\A_n$ from $0$ to $n$ starting at the origin, the {\em peak set} of $P$, denoted $\Peak(P)$, is the set of labels of the vertices that are peaks in $P$. Let $\Peak^*(P)=\Peak(PE)$, where $PE$ is the path obtained from $P$ by appending a step $E$ at the end. Note that 
$$\Peak^*(P)=\begin{cases} \Peak(P)\cup\{n\} & \mbox{if $P$ ends with an $N$ step,} \\ \Peak(P) & \mbox{otherwise}.\end{cases}$$

Given a $P\in\A_n$, and thinking of it as a the south east edge of a Young diagram $\lambda$, we define its \emph{hook decomposition} $\hd(P)=\{i_1,i_2,\dots,i_k\}$ as follows.
The number of entries $k$ is the length of the side of the Durfee square of $\lambda$, that is, the largest value such that $\lambda_k\ge k$. The largest entry $i_k$ is the number of boxes in the largest hook of $\lambda$, which consists of the first column and first row of its Young diagram. Now remove the largest hook of $\lambda$ and define $i_{k-1}$ to be the number of boxes in the largest hook of the remaining Young diagram. Similarly, the remaining entries $i_j$ are defined recursively by peeling off hooks in the Young diagram. 

Define
$$\hd^*(P)=\begin{cases} \hd(P)\cup\{n\} & \mbox{if $P$ begins with an $N$ step,} \\ \hd(P) & \mbox{otherwise}.\end{cases}$$

\begin{lemma}\label{lem:hdpeak}
There is a bijection $\bijhd:\Y_{a,b}\rightarrow\Y_{a,b}$ such that $$\Peak(P)=\hd(\bijhd(P)) \quad \mbox{and}\quad \Peak^*(P)=\hd^*(\bijhd(P)).$$
\end{lemma}{}{

\begin{proof}
The map $\bijhd$ that we describe here is a generalization of the bijection $\psi^{-1}$ from~\cite[Lemma~3.5]{BBES} to arbitrary rectangles.
Given $P\in\Y_{a,b}$, let $(x_j,y_j)$, $1\le j\le k$ the coordinates of its $k$ peaks, where $0\le x_1<\dots<x_k\le b-1$ and $1\le y_1<\dots<y_k\le a$.
Then $$g(P)=R_aR_{a-1}\dots R_1S_1S_2\dots S_b,$$ 
where $R_{y_j}=E$ for $1\le j\le k$ and $R_i=N$ otherwise,
and $S_{x_j+1}=N$ for $1\le j\le k$ and $S_i=E$ otherwise.

Clearly, $\bijhd(P)\in\Y_{a,b}$ because it has $a-k+k$ $N$ steps and $b-k+k$ $E$ steps.
It is easy to check from the construction that $$\hd(\bijhd(P))=\{x_1+y_1,x_2+y_2,\dots,x_k+y_k\}=\Peak(P).$$
Also, the last step of $P$ is an $N$ if and only if the first step of $\bijhd(P)$ is an $N$, and so $\hd^*(\bijhd(P))=\Peak^*(P)$ as well.

To see that the map $\bijhd$ is a bijection, note that for any path $Q=R_aR_{a-1}\dots R_1S_1S_2\dots S_b\in\Y_{a,b}$, the positions of the $N$ and $E$ steps can be used to determine the coordinates $(x_j,y_j)$ of the peaks of the path $\bijhd^{-1}(Q)$.
\end{proof}

With some abuse of notation, we also denote by $\bijhd$ the bijection from $\A_n=\bigcup_{a+b=n}\Y_{a,b}$ to itself.

By Theorem~\ref{cara}, $\CI_{2n}(321)$ is in bijection with the family of all subsets of $[n]$. 
Let $\bijsubs$ be the bijection between $\CI_{2n}(321)$ and $\A_n$ given at the beginning of Section \ref{main:results} together with the interpretation of subsets of $[n]$ as lattice paths described above. For $\pi\in \CI_{2n}(321)$, we have
\beq\label{eq:despeak}\Des^+(\pi)=\Peak^*(\bijsubs(\pi)).\eeq
In particular, $\des^+(\pi)=|\Peak^*(\bijsubs(\pi))|$, and $\maj^+(\pi)=\sum_{i\in\Peak^*(\bijsubs(\pi))} i$.

\begin{proof}[Bijective proof of Theorem~\ref{altromodo}] By Equation~\eqref{eq:despeak},
$$\sum_{\pi\in \CI_{2n}(321)}q^{\des^+(\pi)}=\sum_{P\in\A_n} q^{|\Peak^*(P)|}.$$
To find the coefficient of $q^k$, it is enough to count the number of paths $P\in\A_n$ with $|\Peak^*(P)|=k$.
These are precisely paths of the form
$$P=E^{i_1}N^{j_1}NEE^{i_2}N^{j_2}NE\dots NEE^{i_{k}}N^{j_{k}}NE^{i_{k+1}},$$
where $i_\ell,j_\ell\ge0$ for all $\ell$, and $\sum_\ell i_{\ell}+\sum_\ell j_{\ell}=n-2k+1$.
Thus, the number of such paths is the number of ways to put $n-2k+1$ balls into $2k+1$ bins, which is
$\binom{n+1}{2k}$.
\end{proof}

\begin{proof}[Bijective proof of Theorem~\ref{initial}]
Consider the composition 
$$\CI_{2n}(321)\stackrel{\bijsubs}{\longrightarrow}\A_n\stackrel{\bijhd}{\longrightarrow}\A_n.$$
For $\pi\in \CI_{2n}(321)$, Equation~\eqref{eq:despeak} and Lemma~\ref{lem:hdpeak} imply that
$$\Des^+(\pi)=\Peak^*(\bijsubs(\pi))=\hd^*(\bijhd(\bijsubs(\pi)).$$
Note that for $P\in\A_n$, we have $\sum_{i\in\hd(P)}i=\area(P)$, the area of the Young diagram of which $P$ is the south east border. Thus, by the definition of $\hd^*$,
\beq\label{eq:hd*}\sum_{i\in\hd^*(P)}i=\begin{cases} \area(P)+n & \mbox{if $P$ begins with an $N$ step,} \\ \area(P) & \mbox{otherwise}.\end{cases}\eeq
It follows that
$$\sum_{\pi\in \CI_{2n}(321)}q^{\maj^+(\pi)}=\sum_{P\in\A_n} q^{\sum_{i\in\hd^*(P)}i}=\sum_{P\in\A_n} q^{\area(P)}+(q^n-1)\sum_{Q\in\A_{n-1}} q^{\area(Q)},$$
by separating paths that begin with an $N$ step and writing them as $P=NQ$, with $Q\in\A_{n-1}$ and $\area(P)=\area(Q)$. By definition of the $q$-binomial coefficients, we get
$$\sum_{\pi\in \CI_{2n}(321)}q^{\maj^+(\pi)}=\sum_{h=0}^{n} {n \choose h}_q+(q^n-1)\sum_{h=0}^{n-1} {n-1 \choose h}_q.$$

To obtain the equivalent expression $\sum_{h=0}^{n} q^{n-h}{n \choose h}_q$, we need one more bijection.
For any $P\in\Y_{n-h,h}$, let $P'\in\Y_{n-h,h}$ be the path obtained from $P$ by moving the first step of $P$ to the end. Then
$$\area(P')=\begin{cases} \area(P)+h & \mbox{if $P$ begins with an $N$ step,} \\ \area(P)-(n-h) & \mbox{otherwise}.\end{cases}$$
Combining this with equation~\eqref{eq:hd*}, we see that $\sum_{i\in\hd^*(P)}i=\area(P')+n-h$. Thus,
$$\sum_{\pi\in \CI_{2n}(321)}q^{\maj^+(\pi)}=\sum_{h=0}^n \sum_{P\in\Y_{n-h,h}} q^{\sum_{i\in\hd^*(P)}i}=\sum_{h=0}^n \sum_{P'\in\Y_{n-h,h}} q^{\area(P')+n-h}= \sum_{h=0}^n q^{n-h}{n \choose h}_q.$$
\end{proof}

\section{Connections with the hyperoctahedral group}\label{sec:typeB}

 Recall that the hyperoctahedral group $B_n$ is the set of bijections $\pi$ from the set \newline $\{-n,\ldots,-2,-1,1,2,\ldots,n\}$ to itself such that $\pi(-i)=-\pi(i)$ for $1\le i\le n$. In particular, $\pi$ is described by the sequence $\pi=\pi(1)\ldots\pi(n)$, sometimes called a signed permutation. We denote its absolute value by $|\pi|=|\pi(1)|\ldots|\pi(n)|\in\S_n$. It will be convenient to denote negative entries $-a$ by $\bar{a}$.

 The group of centrosymmetric permutations $\C_{2n}$ corresponds bijectively to
the hyperoctahedral group $B_n$ via the map $\Theta: \C_{2n}\to B_n$ that associates a permutation $\pi\in \C_{2n}$ to the signed permutation $\tilde\pi=\Theta(\pi)\in B_n$ defined by
$$
 \tilde\pi(i)=\begin{cases}\pi(n+i)-n& \textrm{if } \pi(n+i)> n\\ \pi(n+i)-n-1 & \textrm{otherwise,} \end{cases}
 $$
 for $1\le i\le n$ (a similar bijection appears in~\cite{Egge}).

For example, if $\pi = 24863157\in\C_{2n}$, then $\Theta(\pi) = \bar2\bar413\in B_n$.
Drawing a permutation $\pi\in\S_n$ as an $n\times n$ array with a marker in column $i$ and row $\pi(i)$ for each $i$,
where rows and columns are labeled by $1,2,\dots,n$ starting from the bottom left,
the operation $\Theta$ amounts to relabeling the rows and columns by $-n,\dots,-1,1,\dots,n$, as shown in Figure~\ref{fig:theta}. 
From this description, it is immediate that $\pi$ is an involution if and only if so is $\tilde\pi$ (meaning that $\tilde{\pi}(\tilde{\pi}(i))=i$ for every $i$), since involutions correspond to arrays that are symmetric with respect to the diagonal.

\definecolor{xdxdff}{rgb}{0.49019607843137253,0.49019607843137253,1.}
\definecolor{qqttcc}{rgb}{0.,0.2,0.8}
\definecolor{dcrutc}{rgb}{0.8627450980392157,0.0784313725490196,0.23529411764705882}
\definecolor{zzttqq}{rgb}{0.6,0.2,0.}
\begin{figure}[h]
\centering
\begin{tikzpicture}[line cap=round,line join=round,>=triangle 45,x=0.5cm,y=0.5cm]
\draw[dotted,thin] (1,1)--(9,9);
\fill[color=zzttqq,fill=zzttqq,fill opacity=0.1] (1.,1.) -- (9.,1.) -- (9.,9.) -- (1.,9.) -- cycle;
\fill[color=qqttcc,fill=qqttcc,fill opacity=0.1] (12.470698248735147,0.972177901199009) -- (20.470698248735147,0.972177901199009) -- (20.470698248735147,8.972177901199007) -- (12.470698248735147,8.972177901199007) -- cycle;
\draw [color=zzttqq] (1.,1.)-- (9.,1.);
\draw [color=zzttqq] (9.,1.)-- (9.,9.);
\draw [color=zzttqq] (9.,9.)-- (1.,9.);
\draw [color=zzttqq] (1.,9.)-- (1.,1.);
\draw (2.,1.)-- (2.,9.);
\draw (3.,1.)-- (3.,9.);
\draw (4.,1.)-- (4.,9.);
\draw (5.,1.)-- (5.,9.);
\draw (6.,1.)-- (6.,9.);
\draw (7.,1.)-- (7.,9.);
\draw (8.,1.)-- (8.,9.);
\draw (1.,8.)-- (9.,8.);
\draw (1.,7.)-- (9.,7.);
\draw (1.,6.)-- (9.,6.);
\draw (1.,5.)-- (9.,5.);
\draw (1.,4.)-- (9.,4.);
\draw (1.,3.)-- (9.,3.);
\draw (1.,2.)-- (9.,2.);
\draw (1.,1.)-- (9.,1.);
\draw (1.,9.)-- (9.,9.);
\draw (1.,9.)-- (1.,1.);
\draw (9.,9.)-- (9.,1.);
\draw [color=qqttcc] (12.470698248735147,0.972177901199009)-- (20.470698248735147,0.972177901199009);
\draw [color=qqttcc] (20.470698248735147,0.972177901199009)-- (20.470698248735147,8.972177901199007);
\draw [color=qqttcc] (20.470698248735147,8.972177901199007)-- (12.470698248735147,8.972177901199007);
\draw [color=qqttcc] (12.470698248735147,8.972177901199007)-- (12.470698248735147,0.972177901199009);
\draw[dotted,thin] (12.470698248735147,0.972177901199009)--(20.470698248735147,8.972177901199007);
\draw (13.470698248735147,0.972177901199009)-- (13.470698248735147,8.972177901199007);
\draw (14.470698248735147,0.972177901199009)-- (14.470698248735147,8.972177901199007);
\draw (15.470698248735147,0.972177901199009)-- (15.470698248735147,8.972177901199007);
\draw (16.470698248735147,0.972177901199009)-- (16.470698248735147,8.972177901199007);
\draw (17.470698248735147,0.972177901199009)-- (17.470698248735147,8.972177901199007);
\draw (18.470698248735147,0.972177901199009)-- (18.470698248735147,8.972177901199007);
\draw (19.470698248735147,0.972177901199009)-- (19.470698248735147,8.972177901199007);
\draw (12.470698248735147,7.972177901199007)-- (20.470698248735147,7.972177901199007);
\draw (12.470698248735147,6.972177901199007)-- (20.470698248735147,6.972177901199007);
\draw (12.470698248735147,5.972177901199007)-- (20.470698248735147,5.972177901199007);
\draw (12.470698248735147,4.972177901199008)-- (20.470698248735147,4.972177901199008);
\draw (12.470698248735147,3.972177901199008)-- (20.470698248735147,3.972177901199008);
\draw (12.470698248735147,2.972177901199009)-- (20.470698248735147,2.9721779011990086);
\draw (12.470698248735147,1.972177901199009)-- (20.470698248735147,1.9721779011990088);
\draw (12.470698248735147,0.9721779011990088)-- (20.470698248735147,0.9721779011990088);
\draw (12.470698248735147,8.972177901199009)-- (20.470698248735147,8.972177901199009);
\draw (12.470698248735147,8.972177901199009)-- (12.470698248735147,0.9721779011990088);
\draw (20.470698248735147,8.972177901199009)-- (20.470698248735147,0.9721779011990088);
\begin{scriptsize}
\draw [color=dcrutc] (1.5,5.5) circle (3.5pt);
\draw [color=dcrutc] (2.5,3.5) circle (3.5pt);
\draw [color=dcrutc] (3.5,2.5) circle (3.5pt);
\draw [color=dcrutc] (5.5,1.5) circle (3.5pt);
\draw [color=dcrutc] (4.5,8.5) circle (3.5pt);
\draw [color=dcrutc] (6.5,7.5) circle (3.5pt);
\draw [color=dcrutc] (7.5,6.5) circle (3.5pt);
\draw [color=dcrutc] (8.5,4.5) circle (3.5pt);
\draw [color=xdxdff] (12.970698248735147,5.472177901199009) circle (3.5pt);
\draw [color=xdxdff] (13.970698248735147,3.472177901199009) circle (3.5pt);
\draw [color=xdxdff] (14.970698248735147,2.472177901199009) circle (3.5pt);
\draw [color=xdxdff] (16.970698248735147,1.472177901199009) circle (3.5pt);
\draw [color=xdxdff] (15.970698248735147,8.472177901199007) circle (3.5pt);
\draw [color=xdxdff] (17.970698248735147,7.472177901199009) circle (3.5pt);
\draw [color=xdxdff] (18.970698248735147,6.472177901199009) circle (3.5pt);
\draw [color=xdxdff] (19.970698248735147,4.472177901199009) circle (3.5pt);
\end{scriptsize}
\node at (5,9.5){$\pi$};
\node at (16.5,9.5){$\tilde{\pi}$};
\node at (0.5,8.5) {8} ;
\node at (0.5,7.5) {7} ;
\node at (0.5,6.5) {6} ;
\node at (0.5,5.5) {5} ;
\node at (0.5,4.5) {4} ;
\node at (0.5,3.5) {3} ;
\node at (0.5,2.5) {2} ;
\node at (0.5,1.5) {1} ;
\node at (1.5,0.5) {1} ;
\node at (2.5,0.5) {2} ;
\node at (3.5,0.5) {3} ;
\node at (4.5,0.5) {4} ;
\node at (5.5,0.5) {5} ;
\node at (6.5,0.5) {6} ;
\node at (7.5,0.5) {7} ;
\node at (8.5,0.5) {8} ;
\node at (12,8.5) {4} ;
\node at (12,7.5) {3} ;
\node at (12,6.5) {2} ;
\node at (12,5.5) {1} ;
\node at (12,4.5) {-1} ;
\node at (12,3.5) {-2} ;
\node at (12,2.5) {-3} ;
\node at (12,1.5) {-4} ;
\node at (13,0.5) {-4} ;
\node at (14,0.5) {-3} ;
\node at (15,0.5) {-2} ;
\node at (16,0.5) {-1} ;
\node at (17,0.5) {1} ;
\node at (18,0.5) {2} ;
\node at (19,0.5) {3} ;
\node at (20,0.5) {4} ;
\end{tikzpicture}
\caption{An involution $\pi=53281764\in\C_8$ and the corresponding signed involution $\tilde{\pi}=\bar{4}32\bar{1}\in B_4$.}
\label{fig:theta}
\end{figure}
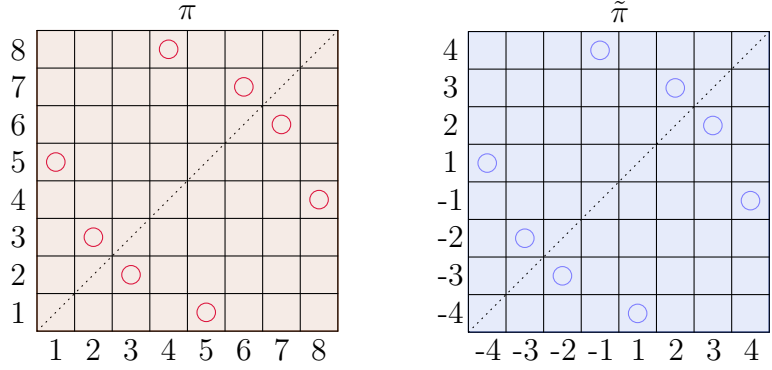

The study of pattern avoidance on the hyperoctahedral group has been carried out by many authors in terms of signed patterns. If $\pi\in B_n$ and $\tau\in B_k$,
$k\leq n$, we say that $\pi$ contains the pattern $\tau$ if there exists a
sequence of indices $1\leq i_1 < i_2 < \cdots < i_k \leq n$ such that two condition hold:
\begin{itemize}
\item $|\pi(i_1)|\ldots|\pi(i_k)|$ is order-isomorphic to $|\tau|$.
\item $\pi(i_j)$ has the same sign as $\tau_j$ for $1\le j\le k$.
\end{itemize}
We say that $\pi$ avoids $\tau$ if $\pi$ does not contain $\tau$.
For example, the signed permutation $\pi = 6\bar15\bar3\bar24$ avoids the pattern
$\tau=\bar21$ while it contains the pattern $\tau'= 2\bar1$.

 Signed pattern avoidance has received a lot of attention in recent years (see \cite{MW}, \cite{SIM}, and \cite{STE}). The image of the set $\C_{2n}(321)$ under $\Theta$ can be characterized in terms of signed pattern avoidance as follows. We remark that this result appears in \cite{Egge} in a slightly different form.

\begin{prop}\label{glielametto}
The set $\Theta(\C_{2n}(321))$ is the set of the elements in $B_n$ that avoid the six patterns 
$$321,\quad \bar321,\quad 32\bar1,\quad \bar32\bar1,$$
$$1\bar2, \quad \bar1\bar2.$$
\end{prop}

\begin{proof} Consider a permutation $\pi\in \C_{2n}(321)$. Straightforward considerations imply that $\Theta(\pi)$ must avoid the six signed patterns above, since every occurrence of one of those signed patterns in  $\Theta(\pi)$ yields an occurrence of $321$ in $\pi$.

 Conversely, suppose that the permutation $\pi$ contains an occurrence $cba$ of $321$. We want to show that $\Theta(\pi)$ contains at least one of the six patterns above. Denote by $j,k,$ and $l$ the positions in $\pi$ of $c,b,$ and $a$, respectively. Note that $c>b>a$ and $j<k<l$. Without loss of generality (due to the fact that $\pi$ is centrosymmetric), we can assume that either 

\begin{itemize}
\item[i)] $j,k,l>n$, or
\item[ii)] $j\leq n<k,l$.
\end{itemize} 

 In both cases, if $b,a\leq n$, then $\Theta(\pi)$ contains $\bar1\bar2$, so we will assume that $b>n$. Consider first case i). 
If $c,b,a>n$, then $\Theta(\pi)$ contains $321$. Suppose that $c,b>n\geq a$. If $b<2n+1-a$, then $\Theta(\pi)$ contains $1 \bar2$, while  if $b>2n+1-a$, then  $\Theta(\pi)$ contains $32\bar1$ (note that $b\neq 2n+1-a$ since the centrosymmetric condition forces $2n+1-a$ to appear in a position less than $n$).  

 Consider now case ii). Set  $d=2n+1-c$, and note that position of $d$ in $\pi$ is greater than $n$. If $c,b,a>n$, then $\Theta(\pi)$ contains either $\bar321$ (if $d$ precedes $b$ in $\pi$) or $1\bar2$ (otherwise). If $c,b>n\geq a$, then $\Theta(\pi)$ contains either $\bar32\bar1$ (if $d$ precedes $b$ in $\pi$) or  $1\bar2$ (otherwise). This completes the proof. 
\end{proof}

In \cite{STE}, Stembridge gave a characterization of the set $T_n$ of the so-called fully commutative top elements in $B_n$. In his paper, he proves that $T_n$ is precisely the subset of permutations in $B_n$ that avoid the six patterns appearing in Proposition \ref{glielametto}. Hence,
$\Theta(\C_{2n}(321))=T_n$. However, Stembridge does not consider the subset of involutory elements in $T_n$.

 In the literature, many definitions of the descent set and of the major index of a signed permutation can be found (see e.g. \cite{BIJONA}, \cite{biaze}, and \cite{reiner}). The notions of $\Des^+$ and $\maj^+$ coincide with the analogous statistics $\Des_B$ and $\maj$ introduced in \cite{biaze}, in the sense that
$\Des^+(\pi)=\Des_B(\Theta(\pi))$ and $\maj^+(\pi)=\maj(\Theta(\pi))$.

\section{Descents and fixed points in $321$-avoiding involutions}\label{sec:fp}

The distribution of the major index over the set of $321$-avoiding involutions was studied in~\cite{BBES}, where it is shown to be given by the $q$-analogue of the central binomial coefficients. This is proved by constructing a bijection, as stated in the following theorem, which is the main result from~\cite{BBES}.

\begin{theorem}[{\cite[Theorem 3.4]{BBES}}]\label{thm:BBES}
There is a bijection between $\I_n(321)$ and $\Y_{\fn2,\cn2}$ that maps $\Des$ to $\hd$.
\end{theorem}

Let $\fp(\pi)$ be the number of fixed points of $\pi$, that is, elements $i$ such that $\pi(i)=i$.
Using the bijection from Lemma~\ref{lem:hdpeak}, we can generalize Theorem~\ref{thm:BBES} as follows.

\begin{theorem}\label{thm:fp}
Let $b\ge a\ge 0$. There is a bijection
$$\genbij:\{\pi\in\I_{a+b}(321):\fp(\pi)\ge b-a\}\longrightarrow\Y_{a,b}$$
such that if $\genbij(\pi)=\lambda$, then $\Des(\pi)=\hd(\lambda)$.
\end{theorem}

For $b=a$ and $b=a+1$ we recover Theorem~\ref{thm:BBES}.

\begin{proof}
Following the same idea from the proof of Theorem~\ref{thm:BBES} given in~\cite{BBES}, the first step of the construction is the Robinson--Schensted correspondence, which gives a bijection between $\I_{a+b}(321)$ and standard Young tableaux with $a+b$ boxes and at most two rows.
We claim that the number of fixed points of $\pi\in\I_{a+b}(321)$ equals the difference in size of the two rows of the corresponding standard Young tableau. To see this, first note that since $\pi$ is an involution, its excedances and antiexcedances (i.e., positions $i$ such that $\pi(i)<i$) are naturally paired up by symmetry. It follows that the excedance values and fixed points of $\pi$ form a longest increasing subsequence (such a sequence is increasing because $\pi$ avoids $321$, and it is longest because at most one element of each pair excedance-antiexcedance can be in it).
By Schensted's theorem~\cite{Sche}, the length of a longest increasing sequence is the size of the first row of the tableau. Thus, the size of the second row has to be equal to the number of antiexcedances, which equals the number of excedances. It follows that the difference between the sizes of the rows of the tableau is $\fp(\pi)$, as claimed. Alternatively, this fact be easily derived from~\cite{EliPak}.

From the standard Young tableau, we construct a path in $\A_{a+b}$ whose $i$-th step is an $N$ if $i$ is in the top row of the tableau, and an $E$ otherwise. By construction, this path does not go below the diagonal $y=x$, and its number of $N$ steps minus its number of $E$ steps equals $\fp(\pi)$, and so it has $(a+b+\fp(\pi))/2$ $N$ steps and $(a+b-\fp(\pi))/2$ $E$ steps. Additionally, $\Des(\pi)$ becomes the peak set of this path, since $i$ is a descent of $\pi$ if and only if $i$ is in the top row of the corresponding tableau and $i+1$ is in the bottom row.

The next step in the construction consists in matching $N$s and $E$s that face each other in the path, in the sense that the line segment from the midpoint of $N$ to the midpoint of $E$ has slope $1$ and stays below the path. Thinking of the $N$s as opening parentheses and the $E$s as closing parentheses, the matched parentheses properly close each other.
Note that every $E$ step gets matched with an $N$ step, but there are $\fp(\pi)$ unmatched $N$ steps. Construct a new path $P$ by changing the leftmost $(\fp(\pi)+b-a)/2$ unmatched $N$ steps into $E$ steps. The resulting path has $a$ $N$ steps and $b$ $E$ steps, so $P\in\Y_{a,b}$. This step is a bijection between paths in $\A_{a+b}$ not going below $y=x$ and having $\fp(\pi)\ge b-a$ unmatched steps, and $\Y_{a,b}$. It has the property that it preserves the positions of the peaks, and so $\Des(\pi)=\Peak(P)$. The inverse of this step is obtained  by matching $N$s and $E$s that face each other in the path in $\Y_{a,b}$, and then changing all the unmatched $E$s (which necessarily come before the unmatched $N$s, and of which there is at least $b-a$ of them) into $N$s.

Finally, the last step is the bijection $\bijhd$ from Lemma~\ref{lem:hdpeak} applied to $P$. The composition of these bijections produces a path $\lambda:=g(P)\in\Y_{a,b}$ with $\Des(\pi)=\hd(\lambda)$.
\end{proof}

Theorem~\ref{thm:BBES} is used in~\cite{BBES} to prove that
$\sum_{\pi\in\I_n(321)} q^{\maj(\pi)}=\binom{n}{\fn2}_q$. We now get a refinement of this result with respect to the number of fixed points.

\begin{corollary}
For $b\ge a\ge 0$,
$$\sum_{\substack{\pi\in\I_{a+b}(321)\\ \fp(\pi)\ge b-a}} q^{\maj(\pi)}=\binom{a+b}{a}_q.$$
Consequently, if $\ell\equiv n\bmod 2$, then
$$\sum_{\substack{\pi\in\I_{n}(321)\\ \fp(\pi)=\ell}} q^{\maj(\pi)}=\binom{n}{\frac{n-\ell}{2}}_q-\binom{n}{\frac{n-\ell}{2}-1}_q.$$
\end{corollary}

\begin{proof}
Applying the bijection from Theorem~\ref{thm:fp} and using that  $\maj(\pi)=\sum_{i\in\Des(\pi)} i$ and $|\lambda|=\sum_{i\in\hd(\lambda)} i$, we have
$$\sum_{\substack{\pi\in\I_{a+b}(321)\\ \fp(\pi)\ge b-a}} q^{\maj(\pi)}=\sum_{\lambda\in\Y_{a,b}} q^{|\lambda|}=\binom{a+b}{a}_q.$$

For the second formula, note that by the symmetry between excedances and anitiexcedances, the number of fixed points of $\pi\in\I_n(321)$ has the same parity as $n$. Now take $a+b=n$ and subtract the case $b-a=\ell+2$ from the case $b-a=\ell$ in the first formula.
\end{proof}

It is shown in~\cite{BBES} that
$$\sum_{\substack{\pi\in\I_n(321)\\ \des(\pi)=k}}q^{\maj(\pi)}=q^{k^2}\binom{\cn2}{k}_q\binom{\fn2}{k}_q.$$
The refinement that keeps track of the number of fixed points is the following.

\begin{corollary}
For $b\ge a\ge 0$,
$$\sum_{\substack{\pi\in\I_{a+b}(321)\\ \fp(\pi)\ge b-a\\ \des(\pi)=k}} q^{\maj(\pi)}=q^{k^2}\binom{a}{k}_q\binom{b}{k}_q.$$
Consequently, if $\ell\equiv n\bmod 2$, then
$$\sum_{\substack{\pi\in\I_{n}(321)\\ \fp(\pi)=\ell \\ \des(\pi)=k}} q^{\maj(\pi)}=q^{k^2}\left[\binom{\frac{n-\ell}{2}}{k}_q\binom{\frac{n+\ell}{2}}{k}_q-\binom{\frac{n-\ell}{2}-1}{k}_q\binom{\frac{n+\ell}{2}+1}{k}_q\right].$$
\end{corollary}

\begin{proof}
By the bijection in Theorem~\ref{thm:fp}, the left hand side of the first formula is the generating polynomial with respect to area for Young diagrams in $\Y_{a,b}$ whose hook decomposition has $k$ hooks. The right hand side is obtained by decomposing such diagrams in three pieces: the top-left $k\times k$ square, which contributes $q^{k^2}$; a partition in $\Y_{k,a-k}$, which contributes $\binom{a}{k}_q$, and a partition in $\Y_{b-k,k}$, which contributes $\binom{b}{k}_q$.

The second formula follows immediately substituting $a+b=n$ and subtracting the case $b-a=\ell+2$ from the case $b-a=\ell$.
\end{proof}

\end{document}